\documentclass[11pt]{article}
\usepackage{}

\usepackage[total={6in, 8.5in}]{geometry}
\usepackage{amsfonts}
\usepackage{amsthm}
\usepackage{amssymb}
\usepackage{amsmath}
\usepackage{enumerate}
\usepackage[
pdfauthor={ESYZ},
pdftitle={Toughness and spanning trees in K4mf graphs},
pdfstartview=XYZ,
bookmarks=true,
colorlinks=true,
linkcolor=blue,
urlcolor=blue,
citecolor=blue,
bookmarks=false,
linktocpage=true,
hyperindex=true
]{hyperref}

\usepackage{xcolor}
\long\def\ignore#1{}

\let\oldi\ignore

 \oldi{
\newtheorem{THM}{\textbf{Theorem}}[section]
\newtheorem{THMs}{\textbf{Theorem}}[section]
\newtheorem{LEM}[THM]{\textbf{Lemma}}
\newtheorem{CON}[THM]{\textbf{Conjecture}}
\newtheorem{PROP}[THM]{\textbf{Proposition}}
\newtheorem{COR}[THM]{\textbf{Corollary}}
\newtheorem{CORs}{\textbf{Corollary}}[section]
\newtheorem{PRO}[THM]{\textbf{Problem}}

\newtheorem{FAC}{\textbf{Fact}}
\newtheorem{REM}{\textbf{Remark}}
\newtheorem{OPR}{\textbf{Operation}}
\newtheorem{CLA}{\textbf{Claim}}[section]
\setcounter{CLA}{1}
 }

\newtheorem{THM}{Theorem}[section]

\newtheorem{LEM}[THM]{Lemma}
\newtheorem{CON}[THM]{Conjecture}

\newtheorem{COR}[THM]{Corollary}

\theoremstyle{definition}
\newtheorem{DEF}[THM]{Definition}

\newtheorem{case}{Case}
\newtheorem{ccase}{Case}[case]

\begin{document}
\title{Toughness and prism-hamiltonicity of $P_4$-free graphs}

\author{%
	M. N. Ellingham\thanks{Supported by Simons Foundation award no. 429625.} $^{\dagger}$%
	\qquad Pouria Salehi Nowbandegani$^{\dagger}$  \qquad Songling Shan$^{\ddagger}$\\
$^{ \dagger}$ Department of Mathematics, 1326 Stevenson Center,\\
	Vanderbilt University, Nashville, TN 37240\\
$^{\ddagger}$ Department of Mathematics, Illinois State University,\\
Normal, IL 61790\\
	\texttt{mark.ellingham@vanderbilt.edu}\\
	\qquad
	\texttt{pouria.salehi.nowbandegani@vanderbilt.edu}\\
	\qquad
	\texttt{sshan12@ilstu.edu}\\
}

\date{6 January 2019}
\maketitle

 \begin{abstract}
 The  \emph{prism} over  a graph $G$ is the product $G \Box K_2$, i.e.,  the graph obtained
 by taking two copies of $G$ and adding a perfect matching joining the two
 copies of each vertex by an edge.
 The graph $G$ is called \emph{prism-hamiltonian} if it has a hamiltonian prism.
 Jung showed that every $1$-tough $P_4$-free graph with at least three vertices is hamiltonian. In this paper, we extend this to observe that for $k \geq 1$ a $P_4$-free graph has a spanning \emph{$k$-walk} (closed walk using each
 vertex at most $k$ times) if and only if it is $\frac{1}{k}$-tough. As our main result, we show that for the class of $P_4$-free graphs, the three properties of being prism-hamiltonian, having a spanning $2$-walk, and being $\frac{1}{2}$-tough are all equivalent.
 

 \end{abstract}

	
\section{Introduction}

	All graphs considered are simple and finite. Let $G$ be a graph.
	For $S\subseteq V(G)$ the subgraph induced on $V(G)-S$ is denoted by
	$G-S$; we abbreviate $G-\{v\}$ to $G-v$.
	The number of components of $G$ is denoted by $c(G)$.
	The graph is said to be \emph{$t$-tough\/} for a real number
	$t \ge 0$ if $|S|\ge t\cdot c(G-S)$ for each $S\subseteq V(G)$ with $c(G-S)\ge
	2$.
	The \emph{toughness $\tau(G)$\/} is the largest real number $t$ for which
	$G$ is $t$-tough, or $\infty$ if $G$ is complete.  Positive
	toughness implies that
	$G$ is connected.  If $G$ has a hamiltonian cycle it is well known that
	$G$ is $1$-tough.
	
	In 1973, Chv\'atal~\cite{chvatal-tough-c} conjectured that for some
	constant $t_0$, every $t_0$-tough graph is hamiltonian.
	Thomassen (see \cite[p.~132]{STGT-ch6}) showed that there are
	nonhamiltonian graphs with toughness greater than $\frac{3}{2}$.
	Enomoto, Jackson, Katerinis and Saito \cite{MR785651} showed that every
	$2$-tough graph has a $2$-factor ($2$-regular spanning subgraph),
	but also for every $\varepsilon > 0$ constructed $(2-\varepsilon)$-tough graphs with no $2$-factor,
	and hence no hamiltonian cycle.
	Bauer, Broersma and Veldman \cite{Tough-CounterE} constructed
	$(\frac{9}{4}-\varepsilon)$-tough nonhamiltonian graphs for every
	$\varepsilon > 0$.  Thus, any such $t_0$ is at least $\frac{9}{4}$.
	
	There have been a number of papers on toughness conditions that
	guarantee the existence of more general spanning structures in a graph.
	A \emph{$k$-tree} is a tree with maximum degree at most $k$, and a \emph{$k$-walk} is a closed walk with each vertex repeated at most $k$ times.
	A $k$-walk can be obtained from a $k$-tree by visiting each edge of the
	tree twice. Note that a spanning $2$-tree is a hamiltonian path and if a graph has at least three vertices then a spanning $1$-walk is a hamiltonian cycle.
	Win \cite{Win-tough} showed that for $k \ge 3$, every $\frac{1}{k-2}$-tough
	graph has a spanning $k$-tree, and hence a spanning $k$-walk.
	In 1990, Jackson and Wormald made the following
	conjecture.

 \begin{CON}[Jackson and Wormald \cite{JW-k-walks}]\label{k-walkc}
 For each integer $k \ge 2$, every connected $\frac{1}{k-1}$-tough graph
 has a spanning  $k$-walk.
	\end{CON}

	The \emph{prism} over a graph $G$ is the Cartesian product $G \Box K_2$. If $G \Box K_2$ is hamiltonian, we say that $G$ is \emph{prism-hamiltonian}.
	Kaiser et al.~\cite{Kaiser} showed that existence of a
hamiltonian path implies prism-hamiltonicity, which in turn implies
existence of a spanning 2-walk. They gave  examples showing that none of
these implications can be reversed. They also made the following
conjecture, which is analogous to those of Chv\'atal and of Jackson and
Wormald.

\begin{CON}[Kaiser et al.~\cite{Kaiser}]\label{prism}
	There exists a constant $t_1$ such that the prism over any $t_1$-tough graph is hamiltonian.
\end{CON}

Kaiser et al.~also showed that  $t_1$  must be at least $\frac{9}{8}$. 

Our goal is to investigate the conjectures above for $P_4$-free graphs,
which have no induced subgraph isomorphic to a $4$-vertex path.
$P_4$-free graphs are also known as \emph{cographs}.
 Connected $P_4$-free graphs can have arbitrarily low or high toughness:
 $K_m + nK_1$ (where `$+$' denotes join) with $m, n \ge 1$ is
$P_4$-free and has toughness $m/n$ if $n \ge 2$, and $\infty$ if $n =
1$.
 The following
result of Jung shows that Chv\'atal's
conjecture holds for $P_4$-free graphs.

\begin{THM}[Jung {\cite[Theorem 4.4(2)]{MR0491356}}] \label{ham-tough}
	Every $P_4$-free graph with at least three vertices is
hamiltonian if and only if it is $1$-tough.
 \end{THM}

The following corollary of Theorem \ref{ham-tough} shows that a stronger version
of Conjecture \ref{k-walkc} holds for $P_4$-free graphs. The \emph{composition} or \emph{lexicographic product} of  graphs $H$ and $K$, denoted by $H[K]$, is defined as the graph with vertex set $V(H)\times V(K)$ and edge set $\{ (u_1,v_1)(u_2,v_2): u_1u_2\in E(H) \, \mbox{or}\, u_1=u_2 \, \mbox{and}\, v_1v_2\in E(K)\}$.

\begin{COR} \label{kwalk}
	Let $k\geq 1$ be a positive integer. Then a $P_4$-free graph has a spanning $k$-walk if and only if it is $\frac{1}{k}$-tough.
\end{COR}

\begin{proof}
	For necessity, Jackson and Wormald \cite[Lemma
2.1(i)]{JW-k-walks} showed that every graph with a spanning $k$-walk is
$\frac{1}{k}$-tough.  So we just show sufficiency.

	The statement is true for graphs on one or two vertices (note that in those cases a spanning
$1$-walk is not a hamiltonian cycle). Hence, we may assume that $G$ has at least three vertices. Also, we may assume that $k \geq 2$, since the statement is true for $k=1$ by Theorem \ref{ham-tough}. 

Jackson and Wormald~\cite{JW-k-walks} showed that $G$ has a spanning $k$-walk if and only if $G[K_k]$ has a hamiltonian cycle. Now suppose $G$ is a $\frac{1}{k}$-tough $P_4$-graph. It is an easy observation that  $G[K_k]$ is $P_4$-free. Goddard and
	Swart \cite[Theorem 6.1(b)]{Goddard90} showed that $\tau(G[K_k]) = k \tau (G)$, so $\tau(G[K_k]) \geq 1$, and hence $G[K_k]$ is hamiltonian by Theorem \ref{ham-tough}. Therefore, $G$ has a spanning $k$-walk using Jackson and Wormald's result. 
\end{proof}

\begin{THM} \label{1/2-tough}
	A $P_4$-free graph with at least two vertices is prism-hamiltonian if and only if it is $\frac{1}{2}$-tough.
\end{THM}

Jung's result, Theorem \ref{ham-tough},  also confirms that sufficiently
tough $P_4$-free graphs are prism-hamiltonian. However, we show that a
weaker toughness condition is both necessary and sufficient, and it is
the same toughness condition required for $P_4$-free graphs to have a
spanning $2$-walk.
 In a similar way, two of the authors (Ellingham and
Salehi Nowbandegani) \cite{ES} showed that for general graphs having a
spanning $2$-walk and being prism-hamiltonian require  the same
Chv\'{a}tal-Erd\H{o}s condition.
 Note that if $G$ is $P_4$-free, $G \Box K_2$ is not in general
$P_4$-free, so Theorem \ref{ham-tough} cannot directly provide a
necessary and sufficient condition for a $P_4$-free graph to be
prism-hamiltonian.

The following is a simple corollary of Theorem \ref{1/2-tough} and Corollary \ref{kwalk}. 

\begin{COR}\label{threeproperties}
In the class of $P_4$-free graphs with at least two vertices, the properties of being prism-hamiltonian, having a spanning $2$-walk, and being $\frac{1}{2}$-tough are equivalent.
\end{COR}

To confirm the above result we just need to note that the subgraph corresponding to any $2$-walk is  $\frac{1}{2}$-tough,
 and prism-hamiltonicity implies the existence of a spanning $2$-walk.

The proof of Theorem~\ref{1/2-tough} uses an inductive approach, which in general is hard to do for showing results based on toughness.  In Section 2,
we develop tools for proving Theorem~\ref{1/2-tough}, which is then proved in the last section.


We conclude this section with a remark on algorithms.  Corneil, Lerchs
and Stewart Burlingham \cite{CLB1981} showed that hamiltonicity can be
determined in polynomial time for a $P_4$-free graph $G$.  Determining
whether $G$ has a spanning $k$-walk amounts to determining whether the
$P_4$-free graph $G[K_k]$ is hamiltonian.  Every connected $n$-vertex
graph has a spanning $(n-1)$-tree and hence a spanning $(n-1)$-walk, so
we only need to check $G[K_k]$ if $k \le n-2$, and this can be done in
time polynomial in $n$.
 Therefore, determining, for a given $P_4$-free graph $G$ and positive
integer $k$, whether $G$ has a spanning $k$-walk can be done in
polynomial time.
 By Corollary \ref{threeproperties}, determining whether $G$ is
prism-hamiltonian can also be done in polynomial time.

\section{Preliminary results}

In this section, we provide some lemmas for proving Theorem~\ref{1/2-tough}. We define a class of graphs which (when they occur as spanning subgraphs) form a subclass of the SEEP-subgraphs introduced by Paulraja \cite{MR1217391} for finding hamiltonian cycles in prisms.

\begin{DEF}\label{def:SBEP}
	A \emph{simple block EP (SBEP)} graph  $H$ is a connected graph with the following  properties:
	\begin{enumerate}[(i)]
			
			\item each block of $H$ is either an even cycle or an edge, and
		\item each vertex of $H$ is contained in at most
		two blocks.
	\end{enumerate}
\end{DEF}

The edges of an SBEP graph are partitioned into cutedges and cycle edges, and the vertices of an SBEP graph are partitioned into cutvertices and single-block vertices. Note that any SBEP graph has at least two single-block vertices (at least
one in each leaf block, if there are two or more blocks). The following  lemma lets us build a new SBEP subgraph from two given SBEP subgraphs.
 

 \begin{LEM}\label{comb-sbep-lem}
  Suppose $H_1$ and $H_2$ are disjoint SBEP subgraphs of a graph $G$, 
with $x_1y_1 \in E(H_1)$, $x_2y_2 \in E(H_2)$, and $x_1 y_2, x_2 y_1 \in 
E(G)$.
  Then there is an SBEP subgraph $H$ of $G$ with $V(H) = V(H_1) \cup 
V(H_2)$.
  \end{LEM}

\begin{proof}

Each edge $x_1y_1$ or $x_2y_2$ is either a cycle edge or a cutedge. By symmetry, we consider three cases. 

If $x_1y_1$ and $x_2y_2$ are  cycle edges, then define $H=H_1\cup H_2\cup\{x_1y_2, x_2y_1\}-\{x_1y_1, x_2y_2\}$.
If $x_1y_1$ is a cutedge and $x_2y_2$ is a  cycle edge, then define $H=H_1\cup H_2\cup\{x_1y_2, x_2y_1\}-\{x_2y_2\}$.
If $x_1y_1$ and $x_2y_2$ are cutedges, then define $H=H_1\cup H_2\cup\{x_1y_2, x_2y_1\}$.

In each case the two blocks containing $x_1 y_1$ and $x_2 y_2$ are
replaced by a new block that is an even cycle, without changing the
number of blocks to which any vertex belongs.  Therefore, the result $H$
is also an SBEP subgraph.
\end{proof}

\begin{THM}\label{SBEP} 
	Every SBEP graph is prism-hamiltonian.
\end{THM}

\begin{proof}
 Let $G$ be an SBEP graph and let $H=G\Box K_2$, consisting of $G$ and a
copy $G'$ of $G$, with each $v \in V(G)$ joined to its copy $v' \in
V(G')$ by a \emph{vertical edge}.
 We show a stronger statement, that  $H$ has a hamiltonian cycle $C$ such
that each single-block vertex $v$ of $G$ and its copy $v'$ are joined by
a vertical edge of $H$ in $C$.  We show this stronger statement
inductively on the number of blocks in $G$.   
The statement holds if $G$ has a single block, i.e., $G$ is an edge or even cycle.
So we assume that $G$ has a cutvertex $x$.

By Definition \ref{def:SBEP}(ii), $x$ is contained in exactly two blocks $B_1$, $B_2$ of $G$. Hence, $G$ is the union of two connected subgraphs $G_1$
(containing $B_1$) and $G_2$ (containing $B_2$) that have only $x$ in common. Each of $G_1$ and $G_2$ is an SBEP graph in which $x$ is a single-block vertex. By induction $G_1 \Box K_2$ and $G_2 \Box K_2$ have hamiltonian cycles $C_1$ and $C_2$, respectively, using vertical edges corresponding to all
single-block vertices, including $xx'$. Now $(C_1 - xx') \cup (C_2 - xx')$ is the required hamiltonian cycle in $G \Box K_2$.
\end{proof}

Let $G$ be a graph and $S\subseteq V(G)$. The set $S$ is called a \emph{tough-set\/} of $G$ if $S$ is a cutset of $G$ and $\frac{|S|}{c(G-S)}=\tau(G)$.
Let $S$ be a cutset of $G$ and $X\subseteq S$. Define $c(G,S,X)$ to be the number of components of $G-S$ that are adjacent in $G$ to vertices of $X$. If $X_1, X_2, \dots , X_k$ are disjoint nonempty subsets of $V(G)$ then by $G[X_1, X_2, \dots , X_k]$ we mean the $k$-partite subgraph of $G$ with vertex set $X_1 \cup  X_2 \cup \dots \cup X_k$ and edge set $\{ uv \in E(G) | u \in X_i, v \in X_j, 1 \leq  i < j \leq k \}$.

\begin{LEM}\label{tough-set1}
	Let $G$ be a connected $P_4$-free graph and let $S$ be a cutset of $G$ such that each vertex in $S$ is adjacent to at least two distinct components of $G-S$. 
	Then  the following statements are true.
	\begin{enumerate}[(i)]
	    \item For each $u\in S$ and each component $R\subseteq G-S$, if $u$
		is adjacent to one vertex in $R$ then $u$ is adjacent to every vertex in $R$. \label{p4u}
		\item  Let $R$ be a component of $G-S$, and let $G'$
		be obtained from $G$ by contracting $R$ into a single vertex.
		Then $G'$ is $P_4$-free. \label{p4Rcontraction}
		\item  If S is a minimal cutset of $G$, then $G[S, V(G) - S]$ is a complete bipartite subgraph of $G$.\label{min-cutset}
		\item \label{com-bipar} Suppose that $S$ is not a minimal cutset of $G$. There exist a cutset $U\subseteq S$ of $G$, nonempty $X\subseteq S-U$
		and nonempty $Y\subseteq V(G)-S$ such that each of the following holds. 
		\begin{enumerate}[(a)]
		\item  $G[X\cup Y]$ is a component of $G-U$.
		\item $G[U,X,Y]$ is a complete tripartite subgraph of $G$. 
		\end{enumerate}
	\end{enumerate}
\end{LEM}

\begin{proof}

 For \eqref{p4u}, suppose $u$ is adjacent to some but not all vertices of $R$.  Since $R$ is connected there must be $v_1 v_2 \in E(R)$ where $v_1$ is adjacent to $u$ but $v_2$ is not.  We know $u$ is also adjacent to $w$ in another component of $G-S$.  Then $v_2 v_1 u w$ is an induced $P_4$, a contradiction.

 The statement \eqref{p4Rcontraction} follows easily by noting that  
any induced $P_4$ of $G'$ corresponds to an induced $P_4$ of $G$ (using \eqref{p4u} 
if the contracted vertex is contained in the $P_4$). For \eqref{min-cutset}, if $S$ is a minimal cutset then each $u \in S$ is adjacent to every component of $G - S$, and hence, by \eqref{p4u}, to every vertex of every
component of $G - S$.

We now show \eqref{com-bipar} by induction on $|V (G)|$. Let $U_0$ be a minimal cutset of $G$ that is contained in $S$. Every vertex in $U_0$ is adjacent to every vertex in $V(G)-U_0$ by  \eqref{min-cutset}; call this $(\star)$. 
As  $S-U_0\ne \emptyset$, $G-U_0$ has a nontrivial component $G_1$ such that  $S\cap V(G_1)\ne \emptyset$.  Let $S_1=S\cap V(G_1)$. Then $G_1$ consists of $G[S_1]$, the components of $G - S$ adjacent to $S_1$,
and the edges of $G$ between $S_1$ and these components. Hence, each vertex in $S_1$ is adjacent
to at least two components of $G_1 - S_1$ (thus, $S_1$ is a cutset of $G_1$). If $S_1$ is a minimal cutset of $G_1$, then  let  $U=U_0$, $X=S_1$ and $Y=V(G_1)-S_1$.  Then (a) holds by definition of $G_1$ and (b) holds by $(\star)$ and because $G[X, Y] = G_1[S_1, V(G_1)-S_1]$ is complete bipartite by \eqref{min-cutset}.

Otherwise, by induction, with $G_1$ taking the role of $G$ and $S_1$ taking the role of $S$, we find a cutset $U_1\subseteq S_1$ of $G_1$, $X_1\subseteq S_1-U_1$ and $Y_1\subseteq V(G_1)-S_1$ such that  $G_1[X_1\cup Y_1]$ is a component of $G_1-U_1$ and $G_1[U_1,X_1,Y_1]$ is a complete tripartite subgraph of $G_1$.
Let $U=U_0\cup U_1$, $X=X_1$, and $Y=Y_1$.  Clearly, $U\subseteq S$, $X\subseteq S-U$ and $Y\subseteq V(G_1)-S_1\subseteq V(G)-S$.  We claim that $U, X$ and $Y$ satisfy (a) and (b).
 Since $G_1$ is a component of $G - U_0$, every component of $G_1 - U_1$ is a component of $G - U_0 - U_1 = G - U$, so $U$ is a cutset of $G$ and $G_1[X_1 \cup Y_1] = G[X \cup Y]$ is a component of $G - U$. Because $G_1[U_1,X_1,Y_1] = G[U_1,X,Y]$ is a complete tripartite graph and by $(\star)$, we see that 
 $G[U,X,Y]$ is a complete tripartite subgraph of $G$. 
\end{proof}

\begin{LEM}\label{tough-set}
	Let $G$ be a connected  graph and let $S$ be a tough-set of $G$. Suppose $\tau(G)=t \leq 1$.
	Then  the following statements hold.
\begin{enumerate}[(i)]
\item \label{tough-com} For any nonempty $S'\subseteq  S$ with  $S'\ne S$, $S'$ is adjacent in $G$ to at least $|S'|/t+1$ components of $G-S$.  
\item \label{tough-com-s'=s} For any nonempty $S'\subseteq  S$, $S'$ is adjacent in $G$ to at least $|S'|/t$ components of $G-S$.
\item \label{tough-com2} Every vertex of $S$ is adjacent to at least two components in $G-S$. 
\item  \label{tough-contraction} Let $R$ be a component of $G-S$.  If $S$ is a maximal tough-set of $G$, $k$ is a positive integer, and $t \geq \frac{1}{k}$, then $R$ is $\frac{1}{k}$-tough.
\item \label{tough-contraction-whole} Suppose $G$ is $P_4$-free.  Let $R$ be a component of $G-S$, and let $G'$ be obtained from $G$ by contracting $R$ into a single vertex. Then $G'$ is $t$-tough. 
	\end{enumerate}
\end{LEM}

An equivalent way to state the conclusion of \eqref{tough-contraction}
is that $R$  is $(1/\lceil 1/t \rceil)$-tough. We cannot in general
strengthen this to say that $R$ is $t$-tough. For example, suppose that
$p \geq 2$ and $G = ((2p - 2)K_1 \cup  K_{1,2}) + K_p$. It is not
difficult to show that $\tau (G) = \frac{p}{2p-1} $, with maximal
tough-set $S = V (K_p)$, but the component $R = K_{1,2}$ of $G - S$ is
only $\frac{1}{2}$-tough, not $\frac{p}{2p-1} $-tough.

\begin{proof}
	For \eqref{tough-com}, let $S^*=S-S'\neq \emptyset$.  Note that
$|S^*| \geq t \, c(G-S^*)$,  by toughness if $c(G-S^*) \geq 2$, and
because $t \leq 1$ if $c(G-S^*)=1$. Also, $c(G-S^*) \ge c(G-S) - c(G,
S, S')$+1. 
	Then 
	\begin{eqnarray*}
	|S'|&=&|S|-|S^*|\le |S|-t\,c(G - S^*)= t\,c(G-S)-t\,c(G - S^*) \\
	    &\leq& t\,c(G-S)-t(c(G-S)-c(G,S,S')+1)=t\,c(G,S,S')-t.
	\end{eqnarray*}
	implying that $c(G,S,S')\ge |S'|/t+1$. For \eqref{tough-com-s'=s}, use \eqref{tough-com} if $S' \neq S$, and if $S'=S$ we have $c(G-S)=|S|/t$ since $S$ is a tough-set. 

For \eqref{tough-com2}, if $|S|\ge 2$, it follows directly from \eqref{tough-com} by taking $S'$ as singletons. 
If $|S|=1$, then the single vertex of $S$ is adjacent to every component of $G - S$.
		
For \eqref{tough-contraction}, we may assume $R$ is not complete. Let $Q\subseteq V(R)$ be a tough-set of $R$. 
 Since $S$ is a maximal tough-set of $G$, $S\cup Q$ is not a tough-set of $G$, but it is a cutset of $G$. 
Then
$$ |S|+|Q|=|S\cup Q| >  t\,c(G-(S\cup Q)) =t(c(G-S)-1+c(R-Q)).$$
Since $|S| = t\,c(G - S)$, we see that $|Q| > t(c(R - Q) - 1)$, and since
$t \geq \frac{1}{k}$ we have $k|Q| > c(R - Q) - 1$. Because both sides are integers, $k|Q| \geq c(R - Q)$, and
so $R$ is $\frac{1}{k}$-tough.

 Now we prove \eqref{tough-contraction-whole}.
 By \eqref{tough-com2}, Lemma \ref{tough-set1} applies to $G$ and $S$.
 By Lemma
\ref{tough-set1}\eqref{p4Rcontraction}, $G'$ is $P_4$-free.  Let $Q$ be
a tough-set of $G'$ and  $\tau(G') = t'$. We may assume that $t' \leq
1$; otherwise, $t \leq 1 < t'$. 
 Then by \eqref{tough-com2}, Lemma \ref{tough-set1} also applies to $G'$
and $Q$.
Let $v_R$ be the vertex to which $R$ is
contracted.  If $v_R \notin Q$ then $Q$ is also a cutset of $G$ with
$c(G - Q) = c(G' - Q)$. Then 
 $$t' = \frac{|Q|}{c(G'-Q)} = \frac{|Q|}{c(G-Q)} \geq t.$$
 So we may assume $v_R \in Q$. Let $A_1, A_2, \dots, A_a$ be the
components of $G' - Q$ adjacent in $G'$ to $v_R$, where $a \geq 2$ by
\eqref{tough-com2}.  By Lemma \ref{tough-set1}\eqref{p4u} for $G'$ and
$Q$, $v_R$ is adjacent in $G'$ to every vertex of  $A_i$ for all $i$
with $1 \leq
i \leq a$, i.e., $v_R$ is adjacent in $G'$ to every vertex of $X =
\bigcup_{i=1}^{a} V(A_i)$. On the other hand,  all neighbors of $v_R$ in
$G'$ lie in $S$, and hence $X \subseteq S$.

Let $B_1 = v_R, B_2, \dots , B_b$ be the components of $G' - S$ adjacent in $G'$ to vertices of $X$, and
$Y = \bigcup_{i=1}^{b} V(B_i)$. The components of $G - S$ adjacent in $G$ to $X$ are just $R$ and $B_2, \dots , B_b$, i.e., $c(G, S, X) = b$. Now by \eqref{tough-com-s'=s} for $G$ and $S$, we have 
 \begin{equation}\label{eq1}
|Y| \geq b = c(G, S, X) \geq |X|/t.
 \end{equation}
 Suppose $2 \leq i \leq b$. By Lemma \ref{tough-set1}\eqref{p4u} for $G$
and $S$, if $u \in X$ is adjacent in $G$ to some vertex of $B_i$, then
$u$ is adjacent to all vertices of $B_i$. Thus, every vertex of $B_i$ is
adjacent in $G$, and hence in $G'$, to some vertex of $X$. Since $X$ is
the union of components of $G' - Q$, all edges leaving $X$ go to $Q$, so
$V (B_i) \subseteq Q$. Moreover, $V(B_1) = \{v_R\} \subseteq Q$ and
hence $Y \subseteq Q$. 

 Let $Z$ be the set of vertices in all components of $G'-Q$
other than $A_1, A_2, \ldots, A_a$.  Then $Z = V(G') -Q - X$, and there
are no edges of $G'$ from $\{v_R\} \cup X$ to $Z$.
 Thus, there is no edge in $G'$ from $Y$ to $Z$; otherwise, there is an
induced $P_4$ starting at $v_R$ then visiting a vertex of $X$, a vertex
of $Y - \{v_R\}$ (which is nonempty because $|Y| \geq |X|$ by
\eqref{eq1}, and $|X| \geq a \geq 2$) and a vertex of $Z$. Therefore,
$c(G', Q, Y ) = a$, and by \eqref{tough-com-s'=s} for $G'$ and $Q$ we
have  
 \begin{equation}\label{eq2}
|X| \geq a = c(G', Q, Y ) \geq |Y |/t'.
\end{equation}

By \eqref{eq1} and \eqref{eq2}, $tt' \geq 1$, but $t \leq 1$ by
hypothesis and $t' \leq 1$ by assumption, so $t' = t = 1$, and $t' \geq
t$ as required. \end{proof}

\section{Proof of Theorem~\ref{1/2-tough}}

In this section, we prove Theorem~\ref{1/2-tough}. We actually prove a stronger result, of which the following lemma is a
special case.

\begin{LEM}\label{sbep-bipartite}
	If $G = G[X,Y]$ is a complete bipartite graph with $|X| \le |Y| \le
	2|X|$, then $G$ has a spanning SBEP subgraph in which every element of
	$Y$ is a single-block vertex.
\end{LEM}

\begin{proof}
	If $|X|=1$ then $G$ itself is the required subgraph, so suppose that
	$|X| \ge 2$.
	Since $|X| \le |Y|$ there is a cycle $C$ using $X$ and $|X|$
vertices of $Y$.
	Since $|Y| \le 2|X|$, the vertices not in $C$ form a subset of $Y$ of
	size at most $|X|$, so we can add an edge joining each such vertex to a
	distinct vertex of $X$ to obtain the required subgraph.
\end{proof}

The theorem we prove is the following.

\begin{THM}\label{1/2-tough2}
Let $G$ be a connected $P_4$-free graph with at least two vertices. Then $G$ has a spanning SBEP subgraph if and only if $\tau(G)\ge \frac{1}{2}$.
\end{THM}
	
\begin{proof}
	The necessity  is clear, as  any SBEP subgraph contains a
spanning 2-walk and the subgraph corresponding to a 2-walk is
$\frac{1}{2}$-tough. We show sufficiency.
	 We may assume that $t = \tau(G) < 1$, otherwise Theorem~\ref{ham-tough} implies  that $G$ has a hamiltonian cycle, which is a spanning SBEP subgraph.  We prove Theorem~\ref{1/2-tough2}
	 by induction on $|V(G)|$. The result holds if $|V(G)|\le 3$. 
	 So we assume that $|V(G)|\ge 4$. 
	 Let $S\subseteq V(G)$ be a maximal tough-set of $G$.
 By Lemma \ref{tough-set}\eqref{tough-com2}, Lemma \ref{tough-set1}
applies to $G$ and $S$.
 We consider two cases.

	\begin{case} Suppose $G-S$ has a nontrivial component.
	 Let $R$ be a nontrivial component of $G-S$, and let
	 $G'$ be the graph obtained from $G$ by contracting $R$
	 into a single vertex, which has at least two vertices. By Lemma~\ref{tough-set}\eqref{tough-contraction-whole}, the graph $G'$ is $\frac{1}{2}$-tough, and by   Lemma~\ref{tough-set}\eqref{tough-contraction}, the component $R$ is  $\frac{1}{2}$-tough.
	
	 By induction, $G'$ has a spanning SBEP subgraph $T'$
	 and $R$  has a spanning SBEP subgraph $T_R$. Let  $v_R$
	 be the corresponding contracted vertex in $G'$, and let $x, y$ be two single-block vertices  in $T_R$ (any SBEP graph has at least two single-block vertices). 
	 By Lemma~\ref{tough-set1}\eqref{p4u}, the neighbors of $v_R$ in $T'$ are all adjacent in $G$ to
	 the vertices $x, y$.  Therefore, any subgraph of $G'$, or $T'$, can be embedded in $G$ by
	 replacing $v_R$ by either $x$ or $y$.

If $v_R$ is a single-block vertex in $T'$, we embed $T'$ in $G$ with
$x$ replacing $v_R$.  Then $T' \cup T_R$ is a spanning SBEP subgraph of
$G$.
Now suppose $v_R$ is a cutvertex.
Then $v_R$ is contained in exactly two blocks $B_1$, $B_2$ of $T'$.
Hence, $T'$ is the union of two connected subgraphs $T'_1$ (containing
$B_1$) and $T'_2$ (containing $B_2$) that have only $v_R$ in common.
Each of $T'_1$ and $T'_2$ is an SBEP graph in which $v_R$ is a
single-block vertex.  Embed $T'_1$ in $G$ with $x$ replacing $v_R$, and
embed $T'_2$ in $G$ with $y$ replacing $v_R$.
Then $T'_1 \cup T'_2 \cup T_R$ is a spanning SBEP subgraph of $G$.

\end{case}


	  \begin{case} Suppose each component of $G-S$ is a single vertex.
We may assume that $S$ is not a minimal cutset of $G$. For otherwise, $G[S, V(G)-S]$ is complete bipartite by Lemma~\ref{tough-set1}\eqref{p4u}.  Since $G$ is $\frac{1}{2}$-tough and less than $1$-tough, $|S|<|V(G)-S|\le 2|S|$ and so $G[S, V(G)-S]$, and hence $G$, has a spanning SBEP subgraph by Lemma \ref{sbep-bipartite}.
	
 Applying
Lemma~\ref{tough-set1}\eqref{com-bipar}, we find a cutset $U\subseteq S$
of $G$, $X\subseteq S-U$ and $Y\subseteq V(G)-S$ such that $G[X\cup Y]$ is a
component of $G-U$,  and $G[U, X, Y]$ is a complete tripartite subgraph
of $G$.  Consequently, $G[X,Y]$  is a spanning complete bipartite
subgraph of the component $G[X\cup Y]$ of $G-U$.	
	 
By Lemma \ref{tough-set}\eqref{tough-com}, $|Y| = c(G, S, X) \ge
\lceil |X|/t \rceil + 1$.  Let $Y_1$ be a subset of $Y$ of size $\lceil
|X|/t \rceil$, and $Y_2 = Y - Y_1 \ne \emptyset$.  Let $R$ be the
complete bipartite subgraph $G[X, Y_1]$ of $G$, and let $G'=G-V(R)$.

We now show that $G'$ is $\frac{1}{2}$-tough. Assume to the contrary that $t'=\tau(G')<\frac{1}{2}$, so that $t'<t$. 
Let $Q\subseteq V(G')$ be a tough-set of $G'$, and let 
$$Q_1=Q\cap S  \quad \text{and}\quad Q_2=Q\cap (V(G)-S).$$
 By Lemma \ref{tough-set}\eqref{tough-com2}, Lemma \ref{tough-set1}
applies to $G'$ and $Q$.
We consider three cases below. 

\begin{ccase}\label{bothnonempty}
 Suppose that $U-Q\ne \emptyset$ and $Y_2-Q\ne \emptyset$.
Then there is one
component of $G'-Q$ containing all of $U-Q$ and all of $Y_2-Q$, since $G'[U-Q, Y_2-Q]$ is a complete bipartite subgraph of $G[U, X, Y]$. 
Adding back $X$
and $Y_1$ to $G'$ just adds $X$ and $Y_1$ to this component without changing any of
the other components of $G'-Q$,  so 
$$2 \le c(G'-Q) = c(G-Q) \le  |Q|/t$$
 by toughness of
$G$, contradicting $c(G'-Q) = |Q|/t' > |Q|/t$.
\end{ccase}

\begin{ccase}\label{U-Qempty}
 Suppose that $U-Q=\emptyset$.
 Since $G[X\cup Y]$  is a component of $G-U$ and $U\subseteq Q$, there
are no edges of $G$ from $X \cup Y$, or in particular from $Y_1$, to
$V(G')-Q$. 
 Thus, if $Q^*=Q\cup X$, then $G-Q^* = G[(V(G')-Q) \cup Y_1]$ is
$G'-Q$ together with isolated vertices from $Y_1$.
 Hence, $c(G-Q^*) = c(G'-Q)+|Y_1|$.
 Then because $|Q| = t'\,c(G'-Q)$,
$|Y_1| \ge |X|/t$ and $t' < t$  we have that
\begin{eqnarray*}
	|Q^*|&=&|Q|+|X| \;\;\leq\;\; t'\,c(G'-Q)+t|Y_1|\\
	&< & t(c(G'-Q)+|Y_1|) \;\;=\;\; t\,c(G-Q^*),\quad
\end{eqnarray*}
contradicting $G$ being $t$-tough.
\end{ccase}

\begin{ccase}\label{Y2-Qempty}
 Suppose that $Y_2-Q=\emptyset$.
 Then $Y_2 \subseteq Q_2$, so $Q_2 \ne \emptyset$.
 Let $A_1, A_2, \ldots, A_a$ be the components of $G'-Q$ adjacent in
$G'$ to vertices of $Q_2$.
 Given $A_i$, $1 \le i \le a$, there is $w \in V(Q_2)$ adjacent to some
vertex of $A_i$.
 By Lemma \ref{tough-set1}\eqref{p4u} for $G'$ and $Q$, $w$ is adacent to
every vertex of $A_i$, and hence $V(A_i) \subseteq S$.  Let $S_1 =
\bigcup_{i=1}^a V(A_i) \subseteq S$.
 Vertices of $S_1$ can only be adjacent in $G'$ to vertices of $S \cup Q =
S \cup Q_2$.
 Now by Lemma \ref{tough-set}\eqref{tough-com-s'=s} for $G'$ and $Q$, and because $t' <
\frac{1}{2}$, $|S_1| \ge a = c(G',Q, Q_2) \ge |Q_2|/t' > 2|Q_2|$.

Since $X\subseteq S$,  and $G[X\cup Y]$
is a component of $G-U$, we see that all vertices in $S_1\cup X$ together  are adjacent in $G$ to at most 
$|Q_2 \cup Y|=|Q_2|+|Y_1|$ components of $G-S$. 
Therefore, by Lemma \ref{tough-set}\eqref{tough-com-s'=s}, we have
 $|Q_2| + |Y_1| \ge c(G, S, S_1 \cup X) \ge (|S_1| + |X|)/t$.
 But $|X|/t+1 > \lceil |X|/t \rceil = |Y_1|$
 and $|S_1| > 2|Q_2|$, so we get $|Q_2| + |X|/t + 1 > 2|Q_2|/t + |X|/t$,
giving $|Q_2| + 1 > 2|Q_2|/t \ge 2|Q_2|$, from which $|Q_2| < 1$, which
is a contradiction.

 This concludes the proof that $G'$ is $\frac{1}{2}$-tough.

\end{ccase}
\end{case}

Since $\frac{1}{2} \le t < 1$, we have $|X| < |Y_1| = \lceil |X|/t
\rceil \le 2|X|$.  Thus, by Lemma \ref{sbep-bipartite}, the complete
bipartite subgraph $R$ has a spanning SBEP subgraph $T_R$. By induction,
$G'$ has a spanning SBEP subgraph $T'$.
 Let $xy_1 \in E(T_R)$ with $x\in X$ and $y_1\in Y_1$.
 Let $zy_2 \in E(T')$ with $y_2 \in Y_2$; then $z \in U$.
 Then $T_R$ and $T'$ are two disjoint SBEP subgraphs, and
 $zy_1, xy_2 \in E(G)$ because $G[U,X,Y]$ is complete
tripartite. Hence, by  Lemma \ref{comb-sbep-lem} we obtain a spanning
SBEP subgraph of $G$. 
 \end{proof}

Now combining Theorems~\ref{SBEP} and \ref{1/2-tough2}
gives Theorem~\ref{1/2-tough}.

\bibliographystyle{plain}

\end{document}